\documentclass[12pt,reqno]{amsart}

\usepackage{multirow}
\usepackage{hhline}

\usepackage{mathtools}
\usepackage{times}
\usepackage[T1]{fontenc}
\usepackage{mathrsfs}
\usepackage{latexsym}
\usepackage[dvips]{graphics}
\usepackage[titletoc, title]{appendix}
\usepackage{amsmath,amsfonts,amsthm,amssymb,amscd}
\usepackage[dvipsnames]{xcolor}
\usepackage{hyperref}
\usepackage{amsmath}

\usepackage{color}
\usepackage{breakurl}

\usepackage{comment}
\newcommand{\bburl}[1]{\textcolor{blue}{\url{#1}}}

\newtheorem{thm}{Theorem}[section]

\newtheorem{lem}[thm]{Lemma}
\newtheorem{prop}[thm]{Proposition}
\newtheorem{exa}[thm]{Example}

\newtheorem{defi}[thm]{Definition}
\newtheorem{rek}[thm]{Remark}

\DeclareMathOperator{\supp}{supp}

\numberwithin{equation}{section}

\DeclareFontFamily{U}{mathx}{}
\DeclareFontShape{U}{mathx}{m}{n}{<-> mathx10}{}
\DeclareSymbolFont{mathx}{U}{mathx}{m}{n}
\DeclareMathAccent{\widehat}{0}{mathx}{"70}
\DeclareMathAccent{\widecheck}{0}{mathx}{"71}

\begin{document}

\title{Extensions and new characterizations of some greedy-type bases}

\author{Miguel Berasategui}

\email{\textcolor{blue}{\href{mailto: mberasategui@dm.uba.ar}{mberasategui@dm.uba.ar}}}
\address{IMAS - UBA - CONICET - Pab I, Facultad de Ciencias Exactas
y Naturales, Universidad de Buenos Aires, (1428), Buenos Aires, Argentina}

\author{Pablo M. Bern\'a}

\email{\textcolor{blue}{\href{mailto:pablo.berna@cunef.edu}{pablo.berna@cunef.edu}}}
\address{Departmento de Métodos Cuantitativos, CUNEF Universidad, 28040 Madrid, Spain}

\author{H\`ung Vi\d{\^e}t Chu}

\email{\textcolor{blue}{\href{mailto:hungchu2@illinois.edu}{hungchu2@illinois.edu}}}
\address{Department of Mathematics, University of Illinois at Urbana-Champaign, Urbana, IL 61820, USA}

\begin{abstract} 
Partially greedy bases in Banach spaces were introduced by Dilworth et al. as a strictly weaker notion than the (almost) greedy bases. In this paper, we study two natural ways to strengthen the definition of partial greediness. The first way produces what we call the consecutive almost greedy property, which turns out to be equivalent to the almost greedy property. Meanwhile, the second way reproduces the PG property for Schauder bases but a strictly stronger property for general bases.
\end{abstract}



\thanks{The first author was supported by CONICET-PIP 1609 and ANPCyT PICT-2018-04104. The second author was supported
by the Grant PID2019-105599GB-I00/AEI/10.13039/501100011033  (Agencia Estatal de Investigación, Spain) and 20906/PI/18 from
Fundaci\'on S\'eneca (Regi\'on de Murcia, Spain).}

\maketitle

\section{Introduction}

One main goal of approximation theory is to approximate a vector (signal) $x$ using finite linear combinations of basis vectors. To do so, an adaptive, nonlinear approximation algorithm, also called sparse approximations, is desirable, where elements in an approximation are found by greedy steps; that is, largest terms are chosen for the aforementioned linear combinations. This sort of approximation is called the thresholding greedy algorithm (TGA) introduced by Konyagin and Temlyakov \cite{KT1}.  On the application side, the adaptive nature of sparse approximations makes them powerful in processing large data sets in the areas of signal processing and numerical computation. On the theoretical side, the last two decades have witnessed a flurry of papers studying sparse approximations from the functional analysis perspective. 

We can classify bases of a space based on the effectiveness of approximations given by the TGA. Amongst the  most studied in the literature, and ordered by decreasing strength in terms of the TGA's effectiveness, we find greedy bases \cite{KT1}, almost greedy bases \cite{DKKT}, (strong) partially greedy bases \cite{BBL, DKKT} and quasi-greedy bases \cite{KT1}. Almost greedy bases are those in which the TGA is better (up to some constant) than the best approximation via projections on the basis vectors, whereas partially greedy ones are those in which it is similarly better than the best approximation via partial sums. In this paper, we strengthen the concept of partially greedy bases in two natural ways, with the following aims. First, we study whether we can find naturally-defined types of greedy-like bases which lie strictly between strong partially greedy and almost greedy bases - which might be worthy of further research in themselves. Second, we look for characterizations of almost greedy bases in terms of formally weaker definitions, or of (strong) partially greedy bases in terms of formally stronger ones. Our choice of the partially greedy property is due to the fact that the approximations involve linear projections, which are generally easier to handle than non-linear ones. 
The first way of modifying the definition (see Definition~\ref{defi:  consecutive2}) surprisingly produces almost greedy bases, while the second way (see Definition~\ref{defi: superstrong}) reproduces the partially greedy property in the case of Schauder bases but a strictly stronger property in general. 

\color{black}

\section{Main definitions and notation}

\subsection{Bases in quasi-Banach spaces} 

A quasi-Banach space ($p$-Banach space, resp.) $\mathbb{X}$ is a vector space that is complete with respect to a \textbf{quasi-norm} (\textbf{$p$-norm}, resp.) $\|\cdot\|$. A quasi-norm $\|\cdot\|$ on $\mathbb{X}$ satisfies
\begin{enumerate}
	\item[(a)] $\|x\| \ge 0, \forall x\in\mathbb{X}$, and $\|x\| = 0 $ if and only if $x = 0$,
	\item[(b)] $\|ax\| = |a|\|x\|, \forall x\in\mathbb{X}$ and $a\in \mathbb{F}$, 
	\item[(c)] there exists $k > 0$ such that 
	$$\|x + y\|\ \le\ k(\|x\| + \|y\|).$$
\end{enumerate}
On the other hand, if $\|\cdot\|$ is a $p$-norm, then $\|\cdot\|$ satisfies (a), (b), and 
\begin{enumerate}
	\item[(d)] there exists $0<p\le 1$ such that 
	$$\|x+y\|^p\ \le\ \|x\|^p + \|y\|^p, \forall x, y\in\mathbb{X}.$$
\end{enumerate}

By the celebrated Aoki-Rolewicz’s Theorem, any quasi-Banach space is $p$-convex for some $0 < p\le 1$; that is, there exists $\mathbf C > 0$ such that 
$$\left\|\sum_{n = 1}^m x_n\right\|\ \le\ \mathbf C\left(\sum_{n=1}^m \|x_n\|^p\right)^{1/p}, n\in\mathbb{N}, x_n \in\mathbb{X}.$$
As a result, any quasi-Banach space can be renormed to become a $p$-Banach space. We, therefore, assume that our quasi-Banach space $\mathbb{X}$ has been renormed to be a $p$-Banach space.

\begin{prop}\cite[Corollaries 2.3 and 2.4]{AABW}
	Let $\mathbb{X}$ be a $p$-Banach space for some $0 < p\le 1$. Let $y\in \mathbb{X}$ and $(x_n)_{n\in J}\subset\mathbb{X}$ with $J$ finite. Then
	\begin{enumerate}
		\item For any scalars $(a_n)_{n\in J}$ with $0\le a_n\le 1$, 
		$$\left\|y + \sum_{n\in J}a_nx_n\right\|\ \le\ \mathbf A_p\sup_{A\subset J}\left\|y + \sum_{n\in J}x_n\right\|.$$
		\item For any scalars $(a_n)_{n\in J}$ with $|a_n|\le 1$, 
		$$\left\|y + \sum_{n\in J}a_nx_n\right\|\ \le\ \mathbf A_p\sup_{|\varepsilon_n| = 1}\left\|y + \sum_{n\in J}\varepsilon_n x_n\right\|.$$
		\item  For any scalars $(a_n)_{n\in J}$ with $|a_n|\le 1$, 
		$$\left\|\sum_{n\in J}a_nx_n\right\|\ \le\ \mathbf B_p\sup_{A\subset J}\left\|\sum_{n\in J}x_n\right\|.$$
	\end{enumerate}
\end{prop}
Throughout the paper, let $\mathbb{X}$ be a separable and infinite-dimensional quasi-Banach (or $p$-Banach space) over the field $\mathbb{F} = \mathbb{R}$ or $\mathbb{C}$. Let $\mathbb{X}^*$ be the dual space of $\mathbb{X}$. A collection $\mathcal{B} = (e_n)_{n=1}^\infty\subset \mathbb{X}$ is said to be a (semi-normalized) \textbf{basis} of $\mathbb{X}$ if 
\begin{enumerate}
    \item $\mathbb{X} = \overline{[e_n: n\in\mathbb{N}]}$, where $[e_n: n\in\mathbb{N}]$ denotes the span of $(e_n)_{n=1}^\infty$;
    \item there is a unique semi-normalized sequence $(e_n^*)_{n=1}^\infty\subset \mathbb{X}^*$ such that $e_j^*(e_k) = \delta_{j, k}$ for all $j, k\in\mathbb{N}$;
    \item there exist $c_1, c_2 > 0$ such that 
    $$0 \ <\ c_1 := \inf_n\{\|e_n\|, \|e_n^*\|\}\ \le\ \sup_n\{\|e_n\|, \|e_n^*\|\} \ =:\ c_2 \ <\ \infty.$$
\end{enumerate}
If $\mathcal{B}$ also satisfies 
\begin{enumerate}
\item[(4)] $\mathbb{X}^*\ =\ \overline{[e_n^*: n\in \mathbb{N}]}^{w^*},$
\end{enumerate}
then $\mathcal{B}$ is a \textbf{Markushevich basis}. Additionally, if the partial sum operators $S_m(x) = \sum_{n=1}^m e_n^*(x)e_n$ for $m\in \mathbb{N}$ are uniformly bounded, i.e., there exists $\mathbf C > 0$ such that 
\begin{enumerate}
\item[(5)] $\|S_m(x)\|\ \le\ \mathbf C\|x\|, \forall x\in \mathbb{X}, \forall m\in \mathbb{N}$,
\end{enumerate}
then we say $\mathcal{B}$ is a \textbf{Schauder basis}. Given a basis $\mathcal{B}$, we associate each $x\in\mathbb{X}$ with the formal series $\sum_{n\in\mathbb{N}} e_n^*(x)e_n$. 
Since $\mathcal{B}$ and $\mathcal{B}^*$ are both semi-normalized, we have $\lim_{n\rightarrow\infty}|e_n^*(x)| = 0$. 

We mention some notation that appear throughout the paper. Fix $x\in\mathbb{X}$, finite subsets $A$ and $B\subset \mathbb{N}$, and $\varepsilon = (\varepsilon_n)\subset\mathbb{F}^{\mathbb{N}}$, where $|\varepsilon_n| = 1$. Let
\begin{enumerate}
    \item[(a)] $\|x\|_\infty = \max_{n}|e_n^*(x)|$ and $\supp(x) = \{n: e_n^*(x)\neq 0\}$,
    \item[(b)] $P_A(x) := \sum_{n\in A}e_n^*(x)e_n$ and $P_{A^c}(x) := x - P_A(x)$,
    \item[(c)] $1_A : = \sum_{n\in A}e_n\mbox{ and }1_{\varepsilon A} = \sum_{n\in A}\varepsilon_n e_n$,
    \item[(d)] $\mathbf A_p = (2^p-1)^{-1/p}$ for $0< p \le 1$ and $\mathbf B_p = \begin{cases}2^{1/p}\mathbf A_p&\mbox{ if }\mathbb{F} = \mathbb{R},\\ 4^{1/p}\mathbf A_p&\mbox{ if }\mathbb{F} = \mathbb{C}.\end{cases}$
\end{enumerate}

We write $A < B$ to mean that $a< b$ for all $a\in A$ and $b\in B$\footnote{Note that $A>\emptyset$ and $A < \emptyset$ for any $A\subset\mathbb{N}$.}, while $\sqcup_{i\in I} A_i$, for some index set $I$ and sets $(A_i)_{i\in I}$, means that the $A_i$'s are pairwise disjoint. Finally, for a number $a$, 
$$A|_a \ :=\ \{n\in A: n\ge a\}.$$ 
Finally, for every $m\in \mathbb{N}_{0}$, let $\mathcal{I}^{(m)}:=\{A\subset \mathbb{N}: |A|=m\mbox{ and } A \text{ is an interval}\}$,  
$$\mathcal{I}^{\le m}\ :=\ \bigcup_{0\le k\le m}\mathcal{I}^{(k)},\mbox{ and } \mathcal{I}\ :=\ \bigcup_{m\in \mathbb{N}_0} \mathcal{I}^{(m)}.$$

\subsection{Thresholding Greedy Algorithm and greedy-like bases} In 1999, Konyagin and Temlyakov \cite{KT1} introduced the TGA, which since then has been extensively studied by many researchers. For each $x\in\mathbb{X}$, the algorithm chooses the largest coefficients (in modulus) with respect to a basis $\mathcal{B}$. In particular, a set $A\subset \mathbb{N}$ is an $m$-\textbf{greedy set} of $x$ if $|A| = m$ and 
$$\min_{n\in A}|e_n^*(x)|\ \ge\ \max_{n\notin A}|e_n^*(x)|.$$
The corresponding \textbf{greedy sum} is 
$$G_m(x)\ :=\ \sum_{n\in A}e_n^*(x)e_n.$$
Note that for a vector $x$ having two or more coefficients of equal modulus, $m$-greedy sets and greedy sums may not be unique. Let $G(x, m)$ denote the set of all greedy sets of $x$ of size $m$, and let $\Lambda_m(x)$ be the only $B\in G(x,m)$ such that 
$$B\backslash A \ <\ A\backslash B, \forall A\in G(x,m)\backslash\{B\}.$$

The minimal condition to guarantee the convergence of the TGA is quasi-greediness.

\begin{defi}[\cite{KT1}]\label{d3}\normalfont
A basis $\mathcal{B}$ in a quasi-Banach space $\mathbb X$ is quasi-greedy if there exists $\mathbf C > 0$ such that
		$$\|G_m(x)\|\ \le\ \mathbf C\|x\|,\forall x\in\mathbb{X}, m\in\mathbb{N}, \forall G_m(x).$$
		The least such $\mathbf C$ is denoted by $\mathbf C_q$, called the quasi-greedy constant. Also when $\mathcal{B}$ is quasi-greedy, let $\mathbf C_\ell$ be the least constant such that $$\|x-G_m(x)\|\ \le\ \mathbf \mathbf \mathbf C_\ell\|x\|,\forall x\in\mathbb{X}, m\in\mathbb{N}, \forall G_m(x).$$
		We call $\mathbf C_\ell$ the suppression quasi-greedy constant.
\end{defi}
 
The relation between this property and the convergence was given by Wojtaszczkyk in \cite{W}, where the author proved that a basis in a quasi-Banach space is quasi-greedy if and only if
$$\lim_{n\rightarrow\infty}\Vert x-G_n(x)\Vert=0,\; \forall x\in\mathbb X.$$

From another point of view, the strongest property in terms of convergence of the TGA is greediness (\cite{KT1}): a basis is greedy if there exists $C>0$ such that
$$\Vert x-G_m(x)\Vert\leq C\inf_{z\in\mathbb X : \vert\supp(z)\vert\leq m}\Vert x-z\Vert,\; \forall m\in\mathbb N, \forall x\in\mathbb X, \forall G_m(x).$$

\begin{rek}
Konyagin and Temlyakov proved in \cite{KT1} that a basis in a Banach space is greedy if and only if the basis is unconditional and democratic, and the characterization was extended to quasi-Banach spaces in \cite{AABW}. We recall that a basis $\mathcal{B}$ in a quasi-Banach space $\mathbb X$ is $\mathbf K$-unconditional with $\mathbf K > 0$ if for all $N\in\mathbb{N}$, $$\left\|\sum_{n=1}^Na_ne_n\right\|\ \le\ \mathbf K\left\|\sum_{n=1}^N b_n e_n\right\|,$$
whenever $|a_n|\le |b_n|$ for all $1\le n\le N$. Also, a basis is $\mathbf C$-democratic with $\mathbf C>0$ if
$$\Vert 1_A\Vert\leq \mathbf C\Vert 1_B\Vert,\; \forall \vert A\vert\leq\vert B\vert<\infty.$$

\end{rek}
Our paper is concerned with \textbf{almost greedy and partially greedy bases} introduced by Dilworth, Kalton, Kutzarova, and Temlyakov \cite{DKKT}, and \textbf{strong partially greedy bases}, introduced in \cite{BBL}.

\begin{defi}\normalfont \label{defAG} A basis $\mathcal B$ in a quasi-Banach space $\mathbb X$ is almost greedy if there exists $\mathbf C\ge 1$ such that
\begin{equation}\label{e7}\|x-G_m (x)\|\ \le\ \mathbf C\widetilde{\sigma}_m(x), \forall x\in \mathbb{X}, \forall m\in \mathbb{N}, \forall G_m(x).\end{equation}
where $\widetilde{\sigma}_m(x)=\widetilde{\sigma}_m[\mathcal B](x):= \inf\{\|x-P_A(x)\|: |A| = m\}$. If $\mathbf C$ verifies \eqref{e7}, then $\mathcal{B}$ is said to be $\mathbf C$-almost greedy. 
\end{defi}

The main characterization of almost-greedy bases was given in \cite{DKKT} where the authors proved that a basis in a Banach space is almost-greedy if and only if the basis is quasi-greedy and democratic.

\begin{rek}\rm Several papers have studied characterizations of almost greedy bases in terms of other greedy-like properties, see for example \cite[Theorem 3.3]{AA},  \cite[Theorem 1.10]{B2}, \cite[Theorem 1.12]{BC},  \cite[Corollary 4.3]{BL}, \cite[Theorem 5.4]{C1}, \cite[Theorems 3.2 and 3.6]{DKK}, and  \cite[Theorem 3.3]{DKKT}. Also, the extension of the main characterization of almost-greedy bases in terms of quasi-greediness and democracy was extended in \cite{AABW} to the context of quasi-Banach spaces.
\end{rek}

\begin{defi}\normalfont
A basis $\mathcal{B}$ in a quasi-Banach space $\mathbb X$ is 
\begin{enumerate}
\item partially greedy if there exists $\mathbf C > 0$ such that
\begin{equation}\label{e3}\|x-G_m(x)\|\ \le\ \mathbf C\|x-S_m(x)\|, \forall x\in\mathbb{X},\forall m\in\mathbb{N},\forall G_m(x).\end{equation}
\item strong partially greedy if there exists $\mathbf C > 0$ such that
\begin{equation}\label{e4}\|x-G_m(x)\|\ \le\ \mathbf C\min_{0\le n\le m}\|x-S_n(x)\|, \forall x\in\mathbb{X},\forall m\in\mathbb{N}, \forall G_m(x).\end{equation}
\end{enumerate}
\end{defi}

In \cite{DKKT}, the authors characterized partially-greediness for Schauder bases in Banach spaces as those bases verifying quasi-greediness and conservativeness, where the last condition is exactly as democracy but taken $A<B$. Also, the same characterization works for strong partially greedy bases for general Markushevich bases and the extension for quasi-Banach spaces was proved in \cite{B3}.

\subsection{Main results}

The right sides of \eqref{e3} and \eqref{e4} measure the distance between $x$ and the projection of $x$ onto the first consecutive vectors in $\mathcal{B}$. Hence, it is natural to investigate a similar condition without restricting the consecutive projection to these first terms. This motivates us to introduce what we call \textbf{consecutive almost greedy} bases.

\begin{defi}\label{defi: consecutive2}\normalfont
A basis $\mathcal B$ in a quasi-Banach space $\mathbb X$ is said to be consecutive almost greedy (CAG) if there exists $\mathbf C\ge 1$ such that 
$$\|x-G_m(x)\|\ \le\ \mathbf C \widecheck{\sigma}_m(x), \forall x\in\mathbb{X}, \forall m\in\mathbb{N}, \forall G_m(x),$$
where 
\begin{equation}\label{e5}\widecheck{\sigma}_m(x)\ =\ \widecheck{\sigma}_m[\mathcal{B}](x)\ :=\ \inf\left\{\left\|x-P_I(x) \right\| \,:\, I\in \mathcal{I}, |I|= m\right\}.
\end{equation}
The least such $\mathbf C$ is denoted by $\mathbf C_{ca}$.
\end{defi}

It is clear that $\widetilde{\sigma}_m(x)\le \widecheck{\sigma}_m(x)\le \|x-S_m(x)\|, \forall x\in\mathbb{X}, \forall m\in\mathbb{N}$, so we know that
$$\mbox{ almost greedy } \Rightarrow \mbox{ consecutive almost greedy } \Rightarrow \mbox{ partially greedy}.$$

Given a basis $\mathcal{B}=(e_n)_{n\in\mathbb{N}}$ and a bijection $\pi$ on $\mathbb{N}$, let  $\mathcal{B}_{\pi}=(e_{\pi(n)})_{n\in\mathbb{N}}$ be the reordered basis corresponding to $\pi$. Since the property of being a greedy set is not affected by reorderings, it is easy to see that $\mathcal{B}$ is $\mathbf C$-almost greedy if and only if $\mathcal{B}_{\pi}$ is $\mathbf C$-CAG for every bijection $\pi$. 

\begin{thm}\label{t1}Let $\mathcal{B}$ be a basis in a $p$-Banach space $\mathbb{X}$. The following statements are equivalent: 
\begin{enumerate}
\item[(i)] $\mathcal{B}$ is almost greedy. 
\item[(ii)] There is $\mathbf C\ge 1$ such that 
$$\|x-G_m(x)\|\ \le\ \mathbf C\widecheck{\sigma}_k(x), \forall x\in\mathbb{X}, \forall m\in\mathbb{N}, \forall G_m(x), \forall 0\le k\le m.$$
\item[(iii)] $\mathcal{B}$ is CAG.
\end{enumerate}
Moreover, if $\mathcal{B}$ is $1$-CAG and $p=1$, then $\mathcal{B}$ is $1$-almost greedy. 
\end{thm}

Surprisingly, this first result shows that if $\mathcal{B}_{\pi}$ is CAG for just one bijection, then we already have an almost greedy basis (though not necessarily with the same constant.) In Example \ref{E1}, we further show that this result is unexpected because $\widecheck{\sigma}_m(x)$ is not bounded by $\widetilde{\sigma}_m(x)$.

Next, we talk about almost greediness using $1$-dimensional subspaces.  In \cite[Corollary 1.8]{BB}, the authors characterized greedy bases as those bases for which there is $\mathbf C > 0$ such that 
$$\|x-G_m(x)\|\ \le\ \mathbf C \inf \{d(x,[1_{A}]): A\subset\mathbb{N}, |A| = m\}, \forall x\in\mathbb{X}, \forall m\in\mathbb{N},\forall G_m(x).$$
A similar characterization was also proven for almost greedy Schauder bases in \cite[Theorem 2.15]{DK}: a Schauder basis is almost greedy if and only if there is $\mathbf C>0$ such that for all $x\in\mathbb{X}$, we have
$$\|x-P_{\Lambda_m(x)}(x)\|\ \le\ \mathbf C d(x,[1_{A}])$$
for all $A\subset \mathbb{N}$, $|A| = m$, either $A > \Lambda_m(x)$ or $A < \Lambda_m(x)$. 
In the next proposition, we shall use approximations involving both intervals and 1-dimensional subspaces to characterize almost greedy bases. 

\begin{prop}\label{1dim} Let $\mathcal{B}$ be a basis of a quasi-Banach space $\mathbb{X}$. The following statements are equivalent: 
\begin{enumerate}
\item[(i)]  $\mathcal{B}$ is almost greedy.
\item[(ii)] There is $\mathbf C > 0$ such that, for every $x\in \mathbb{X}$, $m\in \mathbb{N}$, and $A\in G(x,m)$,
$$\|x-P_A(x)\|\ \le\ \mathbf C \inf\{\|x-t1_{\varepsilon B}\|: t\in \mathbb{R}, |B|\le m,  B\sqcup A, \mbox{ sign }\varepsilon\}.$$

\item[(iii)] There is $\mathbf C>0$ such that, for every $x\in \mathbb{X}$ and $m\in \mathbb{N}$, there is $A\in G(x,m)$ for which
$$\|x-P_A(x)\|\le \mathbf C\inf\{\|x-t 1_{I}\|: t\in \mathbb{R}, I\in \mathcal{I}, I < A \mbox{ or } A < I, |I\cap \supp(x)|\le m \}.$$
\end{enumerate}
\end{prop}

\begin{rek}\normalfont
Note the contrast between Theorem \ref{t1} and Proposition \ref{1dim}: in the former, we limit the length of the intervals, whereas in the latter, we only limit the cardinality of their intersection with the support of $x$. One may wonder whether there is a characterization of almost greedy bases that strengthens Proposition \ref{1dim} by limiting the lengths of the intervals to the cardinality of the relevant greedy sets. Example \ref{example: not1dimintervals} shows that there is no such characterization.
\end{rek}

Finally, we study bases that satisfy a stronger condition than \eqref{e4}. 
\begin{defi}\label{defi: superstrong}\normalfont 
A basis is said to be super-strong partially greedy if there exists $\mathbf C > 0$ such that
\begin{equation}\label{e50}\|x-G_m(x)\|\ \le\ \mathbf C\widehat{\widehat{\sigma}}_m(x),\end{equation}
where 
$$\widehat{\widehat{\sigma}}_n(x)\ =\ \widehat{\widehat{\sigma}}_n[\mathcal{B}](x)\ :=\ \inf_{a_n\in\mathbb{F}}\left\{\left\|x-\sum_{n\in A}a_ne_n\right\|: A\subset \{1, \ldots, m\}\right\}.$$
\end{defi}
For Schauder bases, \eqref{e50} is equivalent to both \eqref{e3} and \eqref{e4}. Indeed, from definitions, 
$$\eqref{e50}\ \Longrightarrow\ \eqref{e4} \ \Longrightarrow\ \eqref{e3}.$$
Hence, we need only to verify that $\eqref{e3} \Longrightarrow \eqref{e50}$, i.e., there exists $\mathbf{D}$ such that
\begin{equation}\label{e51}\|x-S_m(x)\|\ \le\ \mathbf D\widehat{\widehat{\sigma}}_m(x), \forall x\in\mathbb{X}, \forall m\in\mathbb{N}.\end{equation}
Pick $x\in\mathbb{X}$, $m\in \mathbb{N}$, $A\subset\{1, \ldots, m\}$, and $(a_n)_{n\in A}\subset \mathbb{F}$. Since our basis is Schauder, we have
$$\|x-S_m(x)\|\ \le\ (1+\mathbf K)\left\|x-\sum_{n\in A}a_ne_n\right\|,$$
where $\mathbf K$ is the basis constant. Taking the infinum over all $(a_n)_{n\in\mathbb{F}}$ and over all subsets $A\subset\{1,\ldots,m\}$, we have \eqref{e51}.

However, the equivalence does not hold in general if we drop the Schauder condition, which we state as the following theorem. The theorem involves conditional almost greedy bases, which are well-known to exist (\cite[Example 10.2.9]{AK}.)

\begin{thm}\label{t3}
Let $\mathcal{B}$ be a conditional almost greedy basis of a Banach space $\mathbb{X}$. There is a bijection $\pi$ on $\mathbb{N}$ and $x\in\mathbb{X}$ such that 
$$\sup_{m\in\mathbb{N}, G_m(x)}\frac{\|x-G_m(x)\|}{\widehat{\widehat{\sigma}}_m[\mathcal{B}_\pi](x)}\ =\ \infty.$$
Therefore, $\mathcal{B}_\pi$ is strong partially greedy but is not super-strong partially greedy. 
\end{thm}


\section{Consecutive almost greedy bases}

In this section, we show that a basis is consecutive almost greedy if and only if it is almost greedy. For that, we will use an stronger property than democracy.

\begin{defi}[\cite{KT1}]\label{d2}\normalfont A basis in a quasi-Banach space is  super-democractic if there is $\mathbf C > 0$ such that
\begin{equation}\label{e110901} \|1_{\varepsilon A}\|\ \le\ \mathbf C\|1_{\delta B}\|, \end{equation}
for all finite sets $A, B\subset \mathbb{N}$ with $|A| \le |B|$ and all $     \delta, \varepsilon$. Let $\Delta_s$ be the smallest constant for which the above inequality holds. Also, when $\varepsilon\equiv \delta\equiv 1$, we say that $\mathcal B$ is $\Delta$-democratic.
If additionally, we require that $A\sqcup B$, then the basis is said to be disjoint super-democratic. The corresponding disjoint democratic and disjoint super-democratic constants are $\Delta_d$ and $\Delta_{sd}$, respectively.
\end{defi}

The next theorem shows the main characterization of almost greedy bases in quasi-Banach spaces.

\begin{thm}\label{theorem: AGiffQGDem}Let $\mathcal{B}$ be a basis for a quasi-Banach space $\mathbb{X}$. The following statements are equivalent: 
\begin{enumerate}
\item[(i)] $\mathcal{B}$ is almost greedy. 
\item[(ii)] $\mathcal{B}$ is quasi-greedy and democratic. 
\item[(iii)] $\mathcal{B}$ is quasi-greedy and super-democratic. 
\item[(iv)] $\mathcal{B}$ is quasi-greedy and disjoint super-democratic. 
\end{enumerate}
\end{thm}
\begin{rek}
The equivalences are proved in \cite[Theorem 3.3]{DKKT} and \cite[Theorem 6.3]{AABW}, except for the implication disjoint super-democratic $\Longrightarrow$ super-democratic, but this is clear since
$$\|1_{\varepsilon A}\|\ \le\ \Delta_{sd}\|1_{D}\|\ \le\ \Delta_{sd}^2  \|1_{\delta B}\|,$$
for all sets $A, B, D\subset\mathbb{N}$ with $|A| = |D|\le |B|$, $D > A\cup B$, and signs $\varepsilon \delta$. 
\end{rek}

The next lemma was used in \cite{DKK} without a proof and was proved for Banach spaces in \cite[Lemma 2.2]{AA} and for quasi-Banach spaces in \cite[Lemma 6.2]{AABW} under a different formulation. For completeness, we include the proof.  
\begin{lem}\label{remarkAG}
Let $\mathcal{B}$ be a $\mathbf C$-almost greedy basis of a $p$-Banach space $\mathbb X$. Then
$$\|x-G_m(x)\|\ \le\ \mathbf C\min_{0\le k\le m}\widetilde{\sigma}_k(x),\forall x\in\mathbb{X}, \forall m\in\mathbb{N}, \forall G_m(x).$$
\end{lem}

\begin{proof}
Choose $0\le k\le m$. Let $A\subset\mathbb{N}$ with $|A| = k$. We show that $\|x-G_m(x)\|\le \mathbf C\|x-P_A(x)\|$ for an arbitrary $G_m(x)$. Choose $B_N \subset\mathbb{N}$ such that $B_N > N$, $B_N\sqcup A$, and $|B_N| = m-k$. By $\mathbf C$-almost greediness and $p$-convexity, we have
$$\|x-G_m(x)\|^p\ \le\ \mathbf C^p\|x-P_A(x) - P_{B_N}(x)\|^p\ \le\ \mathbf C^p\|x-P_A(x)\|^p + \mathbf C^p\|P_{B_N}(x)\|^p.$$
Let $\alpha_N = \max_{n > N}|e_n^*(x)|$, which approaches $0$ as $N\rightarrow\infty$. Then 
$$\|P_{B_N}(x)\|\ \le\ m^{\frac{1}{p}}\alpha_Nc_2 \ \rightarrow\ 0\mbox{ as } N\ \rightarrow\ \infty.$$
Therefore, $\|x-G_m(x)\|\le \mathbf C\|x-P_A(x)\|$, as desired. 
\end{proof}

\begin{proof}[Proof of Theorem \ref{t1}]
That (i) $\Longrightarrow$ (ii) follows at once by Lemma \ref{remarkAG}, whereas (ii) $\Longrightarrow$ (iii) is immediate. To prove  (iii) $\Longrightarrow$ (i), we use one of the equivalences of Theorem \ref{theorem: AGiffQGDem}.

First, we prove that $\mathcal{B}$ is quasi-greedy: pick $x\in \mathbb{X}$ and $m\in \mathbb{N}$, and let $I_n:=\{n+1,\dots,n+m\}$ for all $n\in \mathbb{N}$. Since $\mathcal{B}^*$ is weak$^*$ null and $\mathcal{B}$ is bounded, we have 
$$\|x-G_m(x)\|\ \le\ \mathbf C_{ca}\|x-P_{I_n}(x)\|\ \xrightarrow[n\to \infty]{}\ \mathbf C_{ca}\|x\|.$$
Hence, $\mathcal{B}$ is $\mathbf C_{\ell}$-suppression quasi-greedy with $\mathbf C_{\ell}\leq \mathbf C_{ca}$.

Next, we prove that $\mathcal{B}$ is democratic: choose finite sets $A, B\subset \mathbb{N}$ with $|A|\le |B|$. Let $I_1, I_2\in \mathcal{I}$ such that $A\subset I_1$ and $|I_2| = |B|$. We have
$$\|1_A\|\ =\ \|(1_{I_1} + 1_{I_2}) -  (1_{I_1\backslash A}+1_{I_2})\|\ \le\ \mathbf C_{ca}\|(1_{I_1} + 1_{I_2}) - 1_{I_1}\|\ =\ \mathbf C_{ca}\|1_{I_2}\|.$$
On the other hand, 
$$\|1_{I_2}\|\ =\ \|(1_{I_2} + 1_B) - 1_B\|\ \le\ \mathbf C_{ca}\|(1_{I_2} + 1_B) - 1_{I_2}\|\ =\ \mathbf C_{ca}\|1_B\|.$$
We have that $\|1_A\|\le \mathbf C_{ca}^2\|1_B\|$ and so, $\mathcal{B}$ is democratic.

Finally, if $\mathbb{X}$ is a Banach space, we show that a $1$-CAG basis is $1$-almost greedy. Fix $x\in \mathbb{X}$ with $\|x\|_{\infty}\le 1$, and $k,j\notin \supp(x)$, $k\neq j$, and signs $\varepsilon_k, \varepsilon_j$. For all $\epsilon>0$, we have 
$$
\|x+\varepsilon_k e_k \|\ =\ \|x+\varepsilon_k e_k+(1+\epsilon)\varepsilon_j e_j - (1+\epsilon)\varepsilon_j e_j\|\ \le\ \|x+(1+\epsilon)\varepsilon_j e_j\|. 
$$
Since $\epsilon$ is arbitrary, it follows that
\begin{equation}\|x+\varepsilon_k e_k \|\ \le\ \|x+\varepsilon_j e_j \|.\label{1almost}\end{equation}
An inductive argument using \eqref{1almost} combined with \cite[Theorem 1.5]{AA} gives that $\mathcal{B}$ is $1$-almost greedy. 
\end{proof}

\begin{exa}\normalfont\label{E1}
	We use the canonical basis of $\ell_1$ to show that $$\sup_{x\in\mathbb{X},m\in\mathbb{N}} \frac{\widecheck{\sigma}_m(x)}{\widetilde{\sigma}_m(x)} \ =\ \infty.$$
	which makes Theorem \ref{t1} quite surprising. We show that for any $C>1$, there are $x\in\mathbb{X}$ and $m\in\mathbb{N}$ satisfying $C \widetilde{\sigma}_m(x) < \widecheck{\sigma}_m(x)$. Indeed, fix $C > 1$ and choose $m\in\mathbb{N}$ and a real scalar $a$ such that $a(m-1) > C$. Consider the following vector:
$$x\ :=\ \left(a, \underbrace{\frac{1}{m^2},\ldots, \frac{1}{m^2}}_{m\mbox{ times}}, \ldots, a, \underbrace{\frac{1}{m^2},\ldots, \frac{1}{m^2}}_{m\mbox{ times}}, 0, 0, \ldots\right).$$
There are $m$ blocks of $a, \underbrace{\frac{1}{m^2},\ldots, \frac{1}{m^2}}_{m\mbox{ times}}$. Removing all the coefficients of magnitude $a$, we have
$$\widetilde{\sigma}_m(x)\ \le\ m^2\cdot \frac{1}{m^2}\ =\ 1.$$
In contrast, by the definition of $\widecheck{\sigma}_m(x)$, we cannot remove more than one coefficient of magnitude $a$; hence, 
$$\widecheck{\sigma}_m(x) \ \ge\ (m-1)a\ > \ C.$$
Therefore, $\widecheck{\sigma}_m(x) > C\widetilde{\sigma}_m(x)$.
\end{exa}


\section{Characterizations of almost greedy bases using intervals and $1$-dimensional subspaces}

In 2017, Bern\'{a} and Blasco \cite{BB} characterized greedy bases, that is, bases where the TGA produces the best possible approximation (see \cite{KT1}), using $1$-dimensional subspaces.  Later, Dilworth and Khurana \cite{DK} obtained an analog for almost greedy bases. Recently, the last two authors of the present paper \cite{BC} allowed the coefficients of a nonzero vector in these $1$-dimensional subspaces to be different and examine whether these characterizations still hold. Below we prove Proposition \ref{1dim} which characterizes almost greedy bases using both intervals and $1$-dimensional subspaces. Furthermore, we show that it is not possible to strengthen the result in a natural way. 

\begin{proof}[Proof of Proposition \ref{1dim}]
(i) $\Longrightarrow$ (ii): Let $\mathcal{B}$ be a $\mathbf C_a$-almost greedy basis. By Theorem \ref{theorem: AGiffQGDem}, $\mathcal{B}$ is $\mathbf C_q$-quasi-greedy, $\mathbf C_{\ell}$-suppression quasi-greedy, and $\Delta_{sd}$-disjoint super-democratic for some $\mathbf C_q, \mathbf C_\ell, \Delta_{sd} > 0$. Assume $B\neq\emptyset$ and $P_A(x)\neq x$. Let $m_1 := |B|$ and $a:=\min_{n\in A}|e_n(x)|$. We consider two cases: if $|t|\le 2a$, pick $A_1\in G(x-t 1_{\varepsilon B}, m_1)$. Since $m_1\le m = |A|$ and 
$$A\ \subset\ \{n\in \mathbb{N}: |e_n^*(x-t 1_{\varepsilon B})|\ge a\},$$
it follows that 
$$b\ :=\ \min_{n\in A_1}|e_n^*(x-t 1_{\varepsilon B})|\ \ge\ a\ \ge\ \frac{|t|}{2}.$$
Now pick $D>A_1\cup B$ with $|D|=m_1$. Note that
\begin{align*}
b\|1_{D}\|&\ =\ \|b1_{D}+P_{A_1}(x-t 1_{\varepsilon B})-P_{A_1}(x-t 1_{\varepsilon B})\|\\
 &\ \le\ \mathbf C_{a}\|P_{A_1}(x-t 1_{\varepsilon B})\|\ \le\  \mathbf C_{a}\mathbf C_{q}\|x-t 1_{\varepsilon B}\|. 
\end{align*}
Hence, 
$$\|t 1_{\varepsilon B}\|\ \le\ 2b \Delta_{sd}\|1_{D}\|\ \le\ 2\Delta_{sd}\mathbf C_{a}\mathbf C_{q}\|x-t 1_{\varepsilon B}\|.$$
On the other hand, if $|t|> 2a$, 
$$
|e_n^*(x-t 1_{\varepsilon B})|\ \ge\ |t|-|e_n^*(x)|\ \ge\ |t| - a\ >\ \frac{|t|}{2}, \forall n\in B. 
$$
Thus, there is $A_2\in G(x-t 1_{\varepsilon B},m_1)$ with 
$$
\min_{n\in A_2}|e_n(x-t 1_{\varepsilon B})|\ \ge\ \frac{|t|}{2}. 
$$
Hence, the same argument as above gives
$$\|t 1_{\varepsilon, B}\|\ \le\  2\Delta_{sd}\mathbf C_{a}\mathbf C_{q}\|x-t 1_{\varepsilon B}\|.$$
Therefore, 
\begin{align*}
\|x-P_A(x)\|^p
&\ \le\ \mathbf C_{\ell}^p\|x\|^p\ \le\ \mathbf C_{\ell}^p\|x-t 1_{\varepsilon B}\|^p+\mathbf C_{\ell}^p\|t 1_{\varepsilon B}\|^p\\
&\ \le\  \mathbf C_{\ell}^p\|x-t 1_{\varepsilon B}\|^p+ ( 2\mathbf C_{\ell}\mathbf C_q \mathbf C_a\Delta_{sd})^p\|x-t 1_{\varepsilon B}\|^p, \end{align*}
from which we obtain (ii) with $\mathbf C=\mathbf C_{\ell}(1+2^p\mathbf C_q^p\mathbf C_a^p\mathbf \Delta_{sd}^p)^{\frac{1}{p}}$.

(ii) $\Longrightarrow$ (iii) is immediate.

(iii) $\Longrightarrow$ (i): By  Theorem~\ref{theorem: AGiffQGDem}, it is enough to prove that $\mathcal{B}$ is quasi-greedy and democratic. To prove the former, fix $x\in \mathbb{X}\backslash\{0\}$, $m\in \mathbb{N}$ and $A\in G(x,m)$. By a standard small perturbation argument and density, given $\epsilon>0$, we can find $y\in \mathbb{X}$ with $\|x-y\|<\epsilon \|x\|$ such that $G(y,m)=\{A\}$. Hence, taking $I$ to be the empty interval, we get 
\begin{align*}
\|x-P_A(x)\|^p&\ \le\ \|x-y\|^p+\|P_A(x)-P_A(y)\|^p+\|y-P_A(y)\|^p\\
&\ \le\ \|x-y\|^p+\|P_A(x)-P_A(y)\|^p+\mathbf C^p\|y-x\|^p+\mathbf C^p\|x\|^p\\
&\ \le\ \epsilon^p(1+\|P_A\|^p+\mathbf C^p)+\mathbf C^p\|x\|^p. 
\end{align*}
Since $\epsilon$ is arbitrary and $A$ is finite, this shows that $\mathcal{B}$ is $\mathbf C$-suppression-quasi-greedy.

To prove that $\mathcal{B}$ is democratic, choose finite sets $A, B\subset \mathbb{N}$ with $0<|A|\le |B|$. Pick $I_1\in \mathcal{I}$ so that $I_1>A\cup B$ and $|I_1|=|B|$, and let $I_2\in \mathcal{I}$ be the smallest interval containing $A$. For every $\epsilon>0$, projecting on $I_2$ gives
$$\|1_{A}\|\ = \ \|1_{A}+(1+\epsilon)1_{I_1}-(1+\epsilon)1_{I_1} \|\ \le\  \mathbf C(1+\epsilon)\|1_{I_1}\|.$$
Similarly, projecting on $I_1$ gives
$$\|1_{I_1}\|\ =\ \|1_{I_1}+(1+\epsilon) 1_{B}-(1+\epsilon)1_{B} \|\ \le\ \mathbf C(1+\epsilon)\|1_{B}\|.$$
Thus, $\mathcal{B}$ is $\mathbf C^2$-democratic. 
\end{proof}

We give an example showing that we cannot replace the condition $|I\cap \supp(x)|\le m$ by $|I|\le m$ in item (iii) of Proposition \ref{1dim}. In fact, even if
\begin{equation}\|x-P_A(x)\|\ \le\ \mathbf C \|x-y\|\label{intervals+1-dim}\end{equation}
for all $x\in \mathbb{X}$, $m\in \mathbb{N}$, $A\in G(x,m)$, and all $y\in \mathbb{X}$ with $\supp(y)\subset I$ for some $I\in \mathcal{I}^{(m)}$ such that $I\sqcup A$, the basis may not even be partially democratic.

\begin{defi}\normalfont
A basis $\mathcal B$ of a quasi-Banach space is $\mathbf C$-partially democratic for some $\mathbf C > 0$ if for every finite set $A\subset\mathbb{N}$, there exists a finite set $D\subset \mathbb{N}$ such that $A\subset D$ and for every $B\subset \mathbb{N}\backslash D$ with $|A| = |B|$, we have
$$\|1_A\|\ \le\ \mathbf C\|1_B\|.$$
\end{defi}

\begin{exa}\label{example: not1dimintervals}\normalfont
Choose $1\le p<q<\infty$. The space $\mathbb{X}=\ell_p\times \ell_q$ with the norm $\|(z,y)\|=\max(\|z\|_{p},\|y\|_{q})$ has an $1$-unconditional  basis for which \eqref{intervals+1-dim} holds but is not partially democratic.
\end{exa}
\begin{proof}
Let $\mathcal{B}_1=(z_n)_{n\in \mathbb{N}}$ and $\mathcal{B}_2=(y_n)_{n\in \mathbb{N}}$ be the canonical unit vector bases of $\ell_p$ and $\ell_q$, respectively. Choose a sequence $(s_k)_{k\in \mathbb{N}}\subset\mathbb{N}$ so that for each $m\in \mathbb{N}$, 
$$s_{m} \ >\ (m+1)^{q/p} \mbox{ and } s_{m+1}\ \ge\ 1+2s_m,$$
and let $\mathcal{B}=(e_n)_{n\in \mathbb{N}}$ be an ordering of the basis $\{(z_i,0),(0,y_j)\}_{i,j\in \mathbb{N}}$ with the following property: for each $n\in \mathbb{N}$, there is $i\in \mathbb{N}$ such that $e_n=(z_i,0)$ if and only if $n=s_k$ for some $k\in \mathbb{N}$. 

Since both $\mathcal{B}_1$ and $\mathcal{B}_2$ are $1$-unconditional, $\mathcal{B}$ is $1$-unconditional. Moreover, $\mathcal{B}$ is not partially democratic because it has one subsequence equivalent to $\mathcal{B}_1$ and one equivalent to $\mathcal{B}_2$. We show that \eqref{intervals+1-dim} holds for $\mathcal{B}$. Pick $x, m, A, I, y$ as in \eqref{intervals+1-dim}.
First note that by $1$-unconditionality, we have
$$\|x-y\|\ \ge\ \|x-P_I(x)\|.$$
Hence, we only need to find $\mathbf C>0$ such that 
\begin{equation}\|x-P_A(x)\|\ \le\ \mathbf C\|x-P_I(x)\|, \label{example:not1intervals: toprove}\end{equation}
and which is independent of the vectors and sets involved. To that end, assume $P_A(x)\neq x$, write $x=(z,y)$, and set
\begin{align*}
A_1&\ :=\ A\cap \{s_k\}_{k\in\mathbb{N}} \mbox{ and }  A_2\ :=\ A\backslash A_1;\\
B_1&\ :=\ I\cap \{s_k\}_{k\in\mathbb{N}}  \mbox{ and }  B_2\ :=\ I\backslash B_1;\\
a&\ :=\ \min_{n\in A}|e_n^*(x)|.
\end{align*}
Now define finite sets $A_1', A_2', B_1', B_2'$ so that 
\begin{align*}
(P_{A_1'}(z),0)&\ =\ P_{A_1}(x);\\
(P_{B_1'}(z),0)&\ =\ P_{B_1}(x);\\
(0,P_{A_2'}(y))&\ = \ P_{A_2}(y);\\
(0,P_{B_2'}(y))&\ =\ P_{B_2}(y). 
\end{align*}
We proceed by case analysis.

Case 1: If $\|P_I(x)\|\le \|x\|/2$, then
$$\|x-P_I(x)\|\ \ge\ \|x\|/2 \ \ge\ \|x-P_A(x)\|/2.$$

Case 2: If $\|P_{B_2'}(y)\|_{q} > \|x\|/2$ and $|B_2'|\le 2|A_2'|$, then 
\begin{align*}
\|x-P_I(x)\|&\ \ge\ \|y-P_{B_2'}(y)\|_{q}\ \ge\ \|P_{A'_2}(y)\|_{q}\ \ge\ a |A_2'|^{\frac{1}{q}}\ \ge\ 2^{-\frac{1}{q}}a|B_2'|^{\frac{1}{q}}\\
&\ \ge\ 2^{-\frac{1}{q}}\|P_{B_2'}(y)\|_{q}\ >\ 2^{-\frac{1}{q}-1}\|x\|\ \ge\ 2^{-\frac{1}{q}-1}\|x-P_A(x)\|.
\end{align*}

Case 3: If $\|P_{B_2'}(y)\|_{q}>\|x\|/2$ and $|B_2'|\le 2|A_1'|$, then
\begin{align*}
\|x-P_I(x)\|&\ \ge\ \|z-P_{B_1'}(z)\|_{p}\ \ge\ \|P_{A_1'}(z)\|_{p}\ \ge\ a |A_1'|^{\frac{1}{p}}\ \ge\ 2^{-\frac{1}{p}}a|B_2'|^{\frac{1}{p}}\\
&\ \ge\ 2^{-\frac{1}{p}}a|B_2'|^{\frac{1}{q}}\ \ge\  2^{-\frac{1}{p}}\|P_{B_2'}(y)\|_{q}\ >\ 2^{-\frac{1}{p}-1}\|x\|\ \ge\ 2^{-\frac{1}{p}-1}\|x-P_A(x)\|.
\end{align*}

Case 4: If $\|P_{B_1'}(z)\|_{p}>\|x\|/2$ and $|B_1'|\le |A_1'|$, then 
\begin{align*}
\|x-P_I(x)\|&\ \ge\ \|z-P_{B_1'}(z)\|_{p}\ \ge\ \|P_{A_1'}(z)\|_{p}\ \ge\ a |A_1'|^{\frac{1}{p}}\ \ge\ a|B_1'|^{\frac{1}{p}}\\
&\ \ge\ \|P_{B_1'}(z)\|_{p}\ >\ \|x\|/2\ \ge\ \|x-P_A(x)\|/2.
\end{align*}

Case 5: If $\|P_{B_1'}(z)\|_{p}>\|x\|/2$, $|B_1'|>|A_1'|$, and $|B_1'|\ge 2$, let $k_0:= \max\{k\in \mathbb{N}: s_k\in B_1\}$. We have 
$$|A_2'|\ \ge \ |B_2'|\ =\ |I\setminus B_1|\ \ge\ s_{k_0}-s_{k_0-1}-1\ \ge\ s_{k_0-1}\ge \ s_{|B_1|-1} \ \ge\  |B_1|^{\frac{q}{p}}.$$
Hence, 
\begin{align*}
\|x-P_I(x)\|&\ \ge\ \|y-P_{B_2'}(y)\|_{q}\ \ge\ \|P_{A_2'}(y)\|_{q}\ \ge\ a |A_2'|^{\frac{1}{q}}\ \ge\ a |B_1|^{\frac{1}{p}}\ \ge\ \|P_{B_1'}(z)\|_{p}\\
&\ \ge\ \|x\|/2\ \ge\ \|x-P_A(x)\|/2.
\end{align*}

Case 6: If $\|P_{B_1'}(z)\|_{p}>\|x\|/2$ and $|B_1'|=1$,
then 
\begin{align*}
\|x-P_A(x)\|\ \le\ \|x\|\ \le\ 2\|P_{B_1'}(z)\|_{p}\ \le\ 2a\ \le\  2 \|x-P_I(x)\|,
\end{align*}
where the last inequality follows from the fact that for each $n\in A$, $|e_n^*(x-P_I(x))|=|e_n^*(x)|\ge a$. \\
Given that
$$\|P_I(x)\|\ =\ \max\{\|P_{B_2'}(y)\|_{q}, \|P_{B_1'}(z)\|_{p}\},$$
we have covered all cases, so the proof is complete. 
\end{proof}


\section{A strong partially greedy basis that is not super-strong partially greedy}

\begin{proof}[Proof of Theorem \ref{t3}]
Since $\mathcal{B} = (e_n)_{n\ge 1}$ is not greedy, one can find $y_1\in \mathbb{X}$, $m_1\in \mathbb{N}$, $A_1\in G(y_1,m_1)$, and $z_1\in \mathbb{X}$ with $|\supp(z_1)|\le m_1$ such that 
$$\|y_1-P_{A_1}(y_1)\|\ >\ 2\|y_1-z_1\|.$$
By scaling, we may assume that $\|y_1\|\le 2^{-1}$. Moreover, by standard small perturbations and density arguments, we may assume further that there is $l_1>m_1$ such that $\supp(y_1)= \{1,\dots,l_1\}$ and that $|\supp(z_1)|=m_1$.\\
Note that for each $m\in \mathbb{N}$, $\mathcal{B}_{m}:=(e_n)_{n> m}$ is a conditional almost greedy basis of $\mathbb{X}_{m}:=\overline{[\mathcal{B}_{m}]}$.  In particular, taking $m=l_1$, we can find $y_2\in \mathbb{X}_{l_1}$ with $\|y_2\|\le 2^{-2}$, $m_2\in \mathbb{N}$, $A_2\in G(y_2,m_2)$, and $z_2\in \mathbb{X}_{l_1}$ with $|\supp(z_2)|= m_2$ such that 
$$\|y_2-P_{A_2}(y_2)\|\ >\ 2^2 \|y_2-z_2\|.$$
As before, we may also assume that there is $l_2>m_2$ such that 
$\supp(y_2)= \{l_1+1,\dots,l_1+l_2\}$.
Moreover, by scaling, we can choose $y_2$ so that 
\begin{align*}
\max_{n\in \supp(y_2)}|e_n^*(y_2)|\ <\ \min_{n\in \supp(y_1)}|e_n^*(y_1)|\mbox{ and } \|y_2\|\ < \ 2^{-2}\|y_1-z_1\|. 
\end{align*}
In this manner, inductively we construct sequences $(y_k)_{k\in \mathbb{N}}$, $(z_k)_{k\in \mathbb{N}}$, $(m_k)_{k\in \mathbb{N}}$, $(l_k)_{k\in \mathbb{N}}$ and $(A_k)_{k\in \mathbb{N}}$ such that, for all $k\in \mathbb{N}$, 
\begin{align}
&1\ \le\ m_k\ <\ l_k;\nonumber\\
&A_k\ \in\ G(y_k,m_k);\nonumber\\
&\|y_k\|\ \le\ 2^{-k}\mbox{ and } \|y_{k+1}\|\ <\  2^{-(k+1)}\min_{1\le j\le k}\|y_j-z_j\|;\label{lemma: notschauder, smallenough}\\
&\|y_k-P_{A_k}(y_k)\|\ >\ 2^{k} \|y_k-z_k\|;\nonumber\\
& \max_{n\in \supp(y_{k+1})}|e_n^*(y_{k+1})|\ <\ \min_{n\in \supp(y_k)}|e_n^*(y_{k})|;\label{lemma: notschauder: allgreedy}\\
& \supp(y_k)\ =\ \left\{\sum_{1\le j\le k-1}l_{j}+1,\dots,\sum_{1\le j\le k}l_{j}\right\}.\label{lemma: notschauder: allgreedy2}
\end{align}
 Define 
$$x\ :=\ \sum_{k\in \mathbb{N}}y_k,$$
and, for each $i\in \mathbb{N}$, set 
$$u_i\ :=\ \sum_{1\le k\le i}y_k,\mbox{ } s_i\ :=\ \sum_{1\le k\le i}l_k,\mbox{ } B_i\ :=\ A_i\cup \{k\in \mathbb{N}: 1\le k\le s_{i-1}\}, $$
where $s_0=0$. Note that for each $i\in \mathbb{N}$, 
\begin{align*}
&\supp(u_i)\ =\ \{1,\dots, s_i\};\\
&B_{i+1}\ =\ \supp(u_{i})\cup A_{i+1}\ \subset\ \supp(u_{i+1}).
\end{align*}
Moreover, it follows from \eqref{lemma: notschauder: allgreedy} and \eqref{lemma: notschauder: allgreedy2} that $B_{i+1}\in G(u_{i+1}, s_i+m_{i+1})= G(x, s_i+m_{i+1}) $ for each $i\in \mathbb{N}$. Note that
\begin{align}
\|u_{i+1}-P_{B_{i+1}}(u_{i+1})\|\ = \ \|y_{i+1}-P_{A_{i+1}}(y_{i+1})\|\ >\ 2^{i+1}\|y_{i+1}-z_{i+1}\|. \label{lemma: notschauder: forreordering}
\end{align}
Now we define the bijection $\pi$: for each $j\in \mathbb{N}$, let $\pi_j$ be a bijection on $\{s_{j-1}+1, \dots, s_j\}$ such that
\begin{align*}
\{\pi_j(k): s_{j-1}+1\le k \le s_{j-1}+m_j\}\ =\ \supp(z_j). 
\end{align*}
and let $\pi(k):=\pi_j(k)$ if $s_{j-1}+1\le k\le s_j$. It follows from this choice that, for each $i\in \mathbb{N}$, 
\begin{align}
\supp_{[\mathcal{B}_{\pi}]}(S_{s_i}(x)+z_{i+1})&\ =\ \supp_{[\mathcal{B}_{\pi}]}(S_{s_i}(u_{i+1})+z_{i+1})\nonumber\\
&\ =\ \{1,\dots, s_i+m_{i+1}\}. \label{lemma: notschauder: forreordering2}
\end{align}
Since greedy sets are preserved by reorderings, we have that $B_{i+1}\in G(u_{i+1}, s_i+m_{i+1})= G(x, s_i+m_{i+1})$ with respect to $\mathcal{B}_{\pi}$ as well. \\
Note also that 
$$S_{s_i}[\mathcal{B}_{\pi}](y)\ =\ S_{s_i}[\mathcal{B}](y), \forall y\in \mathbb{X},\forall i\in \mathbb{N}. $$
Thus, applying \eqref{lemma: notschauder, smallenough} and \eqref{lemma: notschauder: forreordering} to $\mathcal{B}_{\pi}$, we obtain 
\begin{align*}
\|x-P_{B_{i+1}}(x)\|& \ =\ \|x-P_{B_{i+1}}(u_{i+1})\|\ \ge\ \|u_{i+1}-P_{B_{i+1}}(u_{i+1})\|-\|x-u_{i+1}\|\\
&\ \ge\ 2^{i+1}\|y_{i+1}-z_{i+1}\|-\sum_{n\ge i+2}2^{-n}\|y_{i+1}-z_{i+1}\|\ \ge\ 2^{i}\|y_{i+1}-z_{i+1}\|.
\end{align*}
On the other hand, 
\begin{align*}
\|x-S_{s_i}(x)-z_{i+1}\|&\ =\ \|x-S_{s_i}(u_{i+1})-z_{i+1}\|\\
&\ \le\ \|u_{i+1}-S_{s_i}(u_{i+1})-z_{i+1}\|+\|x-u_{i+1}\|\\
&\ \le\ \|y_{i+1}-z_{i+1}\|+\sum_{n\ge i+2}2^{-n}\|y_{i+1}-z_{i+1}\|\\
&\ <\ 2 \|y_{i+1}-z_{i+1}\|. 
\end{align*}
Combining the above inequalities for $i\ge 2$, we obtain 
$$\|x-P_{B_{i+1}}(x)\|\ >\ 2^{i-1}\|x-S_{s_i}(x)-z_{i+1}\|\ \ge\ 2^{i-1}\widehat{\widehat{\sigma}}_{s_i+m_i}[\mathcal{B}_{\pi}](x),$$
and the proof is complete. 
\end{proof}


\ \\
\end{document}